\documentclass[12pt,leqno]{amsart}
\usepackage{tikz}
\usetikzlibrary{automata,positioning}
\usepackage{amssymb,amsthm,amsmath,latexsym}
\usepackage{graphicx} % Required for including images
\graphicspath{{figures/}} % Location of the graphics files
\usepackage{booktabs} % Top and bottom rules for table
\usepackage[font=small,labelfont=bf]{caption} % Required for specifying captions to tables and figures
\usepackage{amsfonts, amsmath, amsthm, amssymb} % For math fonts, symbols and environments
\usepackage{wrapfig} % Allows wrapping text around tables and figures

\newtheorem{theorem}{\sc Theorem}[section]

\newtheorem{prop}[theorem]{\sc Proposition}
\newtheorem{cor}[theorem]{\sc Corollary}

 \newtheorem*{thmB}{Theorem B}
 \newtheorem*{thmC}{Theorem C}
 \newtheorem*{thmD}{Theorem D}
 \newtheorem*{thmE}{Theorem E}
  \newtheorem*{propA}{Proposition A}
  \newtheorem*{prop1}{Proposition}
\title{Intransitive Self-similar Groups}

\author{Alex C. Dantas}
\address{Departamento de Matem\'atica, Universidade de Bras\'ilia,
Brasilia-DF, 70910-900 Brazil}
\email{(Dantas) alexcdan@gmail.br}

\author{Tulio M. G. Santos}
\address{Instituto Federal Goiano,
Campos Belos-GO, 73840-000 Brazil}
\email{(Santos) tulio.gentil@ifgoiano.edu.br}

\author{Said N. Sidki}
\address{Departamento de Matem\'atica, Universidade de Bras\'ilia,
Brasilia-DF, 70910-900 Brazil}
\email{(Sidki) ssidki@gmail.com}


\begin{document}
\maketitle

\begin{abstract}
A group is said to be self-similar provided it admits a faithful
state-closed representation on some regular $m$-tree and the group is said
to be transitive self-similar provided additionally it induces transitive
action on the first level of the tree. A standard approach for constructing
a transitive self-similar representation of a group has been by way of a
single virtual endomorphism of \ the group in question. Recently, it was
shown that this approach when applied to the restricted wreath product $%
\mathbb{Z}\wr \mathbb{Z}$ could not produce a faithful transitive
self-similar representations for any $m\geq 2$ (see, \cite{DS}). In this work we
study state-closed representations without assuming the transitivity
condition. This general action is translated into a set of virtual
endomorphisms corresponding to the different orbits of the action on the
first level of the tree. In this manner, we produce faithful self-similar
representations, some of which are also finite-state, for a number of groups
such as $\mathbb{Z}^{\omega}$, $\mathbb{Z}\wr \mathbb{Z}$ and $(\mathbb{Z} \wr \mathbb{Z}) \wr C_{2}$.
\end{abstract}

\section{Introduction}

Self-similar\ groups have been  given directly as automorphism groups of
some $m$-tree , as is the case of the infinite torsion group of Grigorchuk 
\cite{Gr} and those of Gupta-Sidki \cite{GS}, or constructed to act on such
trees by way of a single virtual endomorphism. However these constructions
necessarily produce transitive self-similar groups, in the sense that the
corresponding state-closed group of automorphisms of the $m$-tree satisfies
the additional property of acting transitively on the first level of the
tree.

Transitive state-closed representations have been studied for the
family of abelian groups, finitely generated nilpotent groups, as well as
for metabelian groups, affine linear groups and arithmetic groups; see \cite%
{BarSid},\cite{BerSid}, \cite{K}, \cite{KS} for more details.

It was shown recently that the group $G=\mathbb{Z}\wr \mathbb{Z}$ fails
to have a faithful transitive state-closed representation on an $m$-tree for
any $m$ \cite{DS}. Yet as we will prove, this group admits a faithful
intransitive state-closed representation on the $3$-tree. Indeed, in this
representation the group is a $3$-letter, $3$-state automata group; it has
the following diagram

\begin{center}
\begin{tikzpicture}[shorten >=3pt,node distance=3.5cm,on grid,auto] 
  \node[state] (e) [] {$e$};
  \node[state] (a) [below left of=e] {$\alpha$};
  \node[state] (g) [below right of=e] {$\gamma$};
 \path [->]
   (e) edge [loop above] node {$0|0, 1|1, 2|2$ } (e)
        (a) edge    [loop left]  node {$1|0$} (a)
        (a) edge              node {$0|1, 2|2$} (e)
        (g) edge              node {$2|2$} (a)
        (g) edge  [loop right]      node {$0|0$} (g)
        (g) edge node [swap] {$1|1$} (e);

\end{tikzpicture}

    Diagram 1
\end{center}

\,

Given a self-similar group $G$ we will prove a number of results about
self-similarity of its over-groups from the following types:

(1) restricted direct product $G^{(\omega )}$ of countably many copies of $%
G$;

(2) restricted wreath product $G\wr K$ where $K$ is finite; restricted
wreath product $A\wr G$ where $A$ is abelian, both finitely and infinitely
generated, and for particular cases where $G$ is abelian.

The two types of groups in item (2) are in accord with Gruenberg's dichotomy
of residually-finite wreath products  \cite{Gru}. Our results extend ones which
have appeared in \cite{BarSid} and \cite{DS}.

Given a state-closed subgroup $G$ of automorphisms of the $m$-tree indexed
by strings from a set $Y$ of size $m$, a $G$-data $\left( \mathbf{m},\mathbf{%
H,F}\right) $ is obtained as follows: suppose $G$ have $s$ orbits $%
Y_{i}=\left\{ y_{i1},...,y_{im_{i}}\right\} $ in its action on $Y$ then $%
\mathbf{m=}\left( m_{1}...,m_{s}\right) $. Define  the
set of subgroups  
\begin{equation*}
\mathbf{H=}\left\{ H_{i}\mid \left[ G:H_{i}\right] =m_{i}\text{ }\left( 
\text{ }1\leq i\leq s\right) \right\} \text{ }
\end{equation*}%
where $H_{i}=Fix_{G}\left( y_{i1}\right) $; define the set of projections 
\begin{equation*}
\mathbf{F}=\left\{ f_{i}:H_{i}\rightarrow G\mid 1\leq i\leq s\right\} \text{.%
}
\end{equation*}
On the other hand, given a group $G$, an $s$- set of subgroups of $G$ 
\begin{equation*}
\mathbf{H=}\left\{ \left( H_{i}\mid \left[ G:H_{i}\right] =m_{i}\text{ }%
\left( \text{ }1\leq i\leq s\right) \right) \right\} \text{,}
\end{equation*}%
\begin{equation*}
\mathbf{m=}\left( m_{1}...,m_{s}\right) ,\text{ }m=m_{1}+...+m_{s}
\end{equation*}
and a set of virtual endomorphisms%
\begin{equation*}
\mathbf{F}=\left\{ f_{i}:H_{i}\rightarrow G\mid 1\leq i\leq s\right\} \text{,%
}
\end{equation*}
we have an abstract $G$-data $\left( \mathbf{m},\mathbf{H,F}\right) $. We
prove reciprocally

\begin{propA}
\label{PA copy(1)} Given a group $G$, $m\geq 1$ and a $G$-data $\left( 
\mathbf{m},\mathbf{H,F}\right) $. Then the data provides a state-closed
representation of $G$ on the $m$-tree with kernel 
\begin{equation*}
\langle K\leqslant \cap _{i=1}^{s}H_{i}\mid K\vartriangleleft
G,K^{f_{i}}\leqslant K,\forall i=1,...,s\rangle \text{,}
\end{equation*}%
called the $\mathbf{F}$-core of $\mathbf{H}$ .
\end{propA}

The partition $\mathbf{m=}\left( m_{1},...,m_{s}\right) $ is called the 
\textit{orbit-type} of the representation.

We apply the Proposition A to obtain families of self-similar groups.

\begin{thmB}
\label{TB} Let $G$ be a self-similar group of degree $m$ and orbit-type $%
\left( m_{1},..,m_{s}\right) $. Then  the following hold.

\begin{itemize}
\item[1)] $G^{(\omega )}$ admits a faithful state-closed representation of
degree $m+1$ and orbit-type $\left( m_{1},..,m_{s},1\right) $; in
particular, for $G=\mathbb{Z}$ , the representation of the group $\mathbb{Z}%
^{(\omega )}$ is of orbit-type $\left( 2,1\right),$ and is in addition
finite-state.

\item[2)] Let $K$ be a regular subgroup of $Sym(\{1,...,s\})$. Then the
restricted wreath product $G\wr K$ admits a faithful state-closed representation of degree $(m_1\cdot m_2\cdot...\cdot m_s) \cdot s.$
\end{itemize}
\end{thmB}

With respect to the first item, for any positive integer $k$, the direct
product group $G^{k}$ is \ self-similar of degree $m$. It was shown in \cite%
{BarSid} that the group $\mathbb{Z}^{(\omega )}$ has a faithful transitive
self-similar representation of degree $2$, moreover and importantly, no such
representation can also be finite-state for any degree.
As a consequence
of the tree-wreathing operation defined by Brunner and Sidki \cite{BS1} and
by using $k$-inflation, the group $\mathbb{Z}\wr \mathbb{Z}$ is self-similar of
orbit-type $(2,1)$; see Diagram 1 . Then, by the second item of Theorem B, it follows that $(\mathbb{Z}\wr \mathbb{Z})\wr C_{2}$ is a self-similar group of degree $4$.
 
%The following result answers a question proposed by A. Woryna in \cite[page 100]{W}.

Restricted wreath products $A\wr G$ for $A$ abelian and $G=%
\mathbb{Z}^{d}$ have been a good source for automata groups. \ The first
instance in this family is the classical Lamplighter group $C_{2}\wr \mathbb{%
Z}$. It was shown in \cite{DS} that if $A\wr \mathbb{Z}^{d}$ $\ $admits a
faithful \textit{transitive} self-similar representation then $A$ is
necessarily a \textit{torsion group of finite exponent}. Also, it was proven
in \cite[Proposition 6.1]{BarSid} that when $B$ is a finite abelian group,
then $B\wr \mathbb{Z}^{d}$ is an automata group of degree $2|B|$. 

We generalize both results as follows

\begin{thmC}
\label{TD copy(1)} Let $A$ be a finitely generated abelian group and $%
B=Tor(A)$. Then  $G=A\wr \mathbb{Z}^{d}$ is an automata group of degree $%
2|B|+4$. In particular, for $A=\mathbb{Z}^{l}$, the degree can be reduced to $4$.
\end{thmC}

The theorem will follow from a general process of concatenation (see Proposition \ref{P3.2}) of
the two cases $G_{1}=B\wr \mathbb{Z}^{d}$ and  $G_{2}=\mathbb{Z}^{l}\wr \mathbb{Z%
}^{d}$. We note that the result for $\mathbb{Z} \wr \mathbb{Z}$ answers positively a question posed by A. Woryna in \cite[page 100]{W}.

It was shown in  \cite{DS1}  that the group $C_{p}\wr \mathbb{Z}^{d}$ where $%
C_{p}$ is cyclic of prime order $p$ and $d\geq 2$  is self-similar of degree $p^{2}$, but that such a group does
not have a  faithful \textit{transitive} state-closed representation of prime degree.
In this context we prove:

\begin{thmD}
\label{TC copy(1)} Let $p$ a prime number then $C_{p}\wr \mathbb{Z}^{2}$ is
a self-similar group of degree $p+1$ of orbit-type $%
\left( p,1\right)$. Indeed, $C_{p}\wr \mathbb{Z}^{2}$ is generated by  $\alpha = (\alpha, \alpha \sigma, ..., \alpha \sigma^{p-1}, \alpha \beta)$, $\sigma = (e, ..., e, \sigma)(0 \, 1 \, ... p-1)$ and $\beta = (e, ..., e, \alpha)$. In particular, the group $C_{2}\wr \mathbb{Z}^{2}$ is
self-similar of degree $3$.
\end{thmD}

Let $f:H\rightarrow G$ be a virtual endomorphism. Set $G_{0}=G$ and $%
G_{n}=G_{n-1}^{f^{-1}}$ for all $n\geq 1$. Define the parabolic subgroup $%
G_{\omega }=\cap _{j\geq 0}G_{j}$. Let $G_{\omega }\setminus G$ denote the
set of right cosets $G_{\omega }$ in $G$, let $A^{(G_{\omega }\setminus
G)}=\{\phi:G_{\omega }\setminus
G \to A \,\, \text{ of finite support}\,\,\}$ and have $g \in G$ act on it by translation. The following result was proven for
transitive self-similar groups in \cite{BarSid}.

\begin{prop1}
(Proposition 6.1) Let $G$ be a transitive self-similar group of degree $m$ and parabolic
subgroup $G_{\omega}$, and let $B$ be a finite abelian group. Then the
extension $B^{(G_{\omega}\setminus G)}\rtimes G$ is transitive
self-similar of degree $|B|m$, and is finite-state whenever $G$ is.
\end{prop1}

Let $G$ be a state-closed group with respect the data $(\mathbf{m}, \mathbf{H%
}, \mathbf{F})$, where $\mathbf{m} = (m_{1}, ...,
m_{s})$, with $m_{1} \geq 2, ..., m_{s} \geq 2$. 
For each $i=1,...,s$ define $G_{i0}=G$, $G_{ij}={(G_{i(j - 1)})}^{f_i^{-1}}$ ($j>0$) and $G_{\omega_i}= \cap_{j\geq 0}G_{ij}.$ With this notation we have:

\begin{thmE}
Let $B$ be a finite abelian group and $G$ be a self-similar group of
orbit-type $(m_{1}, ...,
m_{s})$, with $m_{1} \geq 2, ..., m_{s} \geq 2$. Then the group $$B^{((G_{\omega_1}\setminus G) \times ... \times (G_{\omega_s}\setminus G))}\rtimes G^s$$
is self-similar of orbit-type $(|B|.m_{1}...m_{s}, \, 1)$.
\end{thmE}

A question which has remained open is whether the group $C_{2}\wr (\mathbb{Z%
}\wr \mathbb{Z})$ is self-similar.

\section{Preliminaries}

\subsection{\textbf{Groups acting on rooted }$m$-\textbf{trees.}}

The vertices of a rooted $\ m$-tree $\mathcal{T}_{m}$ are indexed by strings
from an alphabet $Y$ of $m\geq 1$ letters, ordered by $u<v$ provided the
string $v$ is a prefix of $u$. The tree $\mathcal{T}_{m}$ is also denoted as 
$\mathcal{T}\left( Y\right) $; normally, we chose to take $Y=\left\{
0,1,...,m-1\right\} $. Given a group and a representation of it on $\mathcal{%
T}_{m}$ we say both the group and its representation have degree $m$.

The automorphism group $\mathcal{A}_{m}$, or $\mathcal{A}\left( Y\right) $,
of $\mathcal{T}_{m}$ is isomorphic to the restricted wreath product
recursively defined as $\mathcal{A}_{m}=\mathcal{A}_{m}\wr S_{m}$, where $%
S_{m}$ is the symmetric group of degree $m$. An automorphism $\alpha $ of $%
\mathcal{T}_{m}$ has the form $\alpha =(\alpha _{0},...,\alpha _{m-1})\sigma
(\alpha )$, where the state $\alpha _{i}$ belongs to $\mathcal{A}_{m}$ and
where $\sigma :\mathcal{A}_{m}\rightarrow S_{m}$ is the permutational
representation of $\mathcal{A}_{m}$ on $Y$, the first level of the tree $%
\mathcal{T}_{m}$. Successive developments of the automorphisms $\alpha _{i}$
produce $\alpha _{u}$ for all vertices $u$ of the tree. For $k\geq 1$, the
action of $\alpha $ on a string $y_{1}y_{2}...y_{k}\in Y^{k}$ is as follows

\begin{equation*}
\alpha :y_{1}y_{2}...y_{k}\mapsto \left( y_{1}\right) ^{\sigma (\alpha
)}\left( y_{2}...y_{k}\right) ^{\alpha _{y_{1}}}\text{. }
\end{equation*}%
This implies that $\alpha $ induces an automorphism $\alpha _{\mathbf{k}}$
on the $m^{k}$-tree $\mathcal{T}\left( Y^{k}\right) $ and the above action
on $k$-stings gives us a group embedding $\mathcal{A}\left( Y\right)
\rightarrow \mathcal{A}\left( Y^{k}\right) $ which we call $k$-\textit{%
inflation}.

For $\alpha \in \mathcal{A}\left( Y\right) $, the set of automorphisms%
\begin{equation*}
Q(\alpha )=\{\alpha ,\alpha _{0},...,\alpha _{m-1}\}\cup
_{i=0}^{m-1}Q(\alpha _{i})
\end{equation*}%
is called the set of \textit{states} of $\alpha $ and this automorphism is
said to be \textit{finite-state} provided $Q(\alpha )$ is finite. A subgroup 
$G$ of $\mathcal{A}_{m}$ is \textit{state-closed} (or, \textit{self similar}%
) if $Q(\alpha )$ is a subset of $G$ for all $\alpha $ in $G$. More
generally, a subgroup $G$ of $\mathcal{A}_{m}$ is $k$th\textit{-level} 
\textit{state-closed} if for all $\alpha \in G$, the states $\alpha _{u}$ belong to 
$G $ for all strings $u$ of length $k$; we see then that the $k$-inflation of 
$G $ is state-closed. A group which is finitely generated, state-closed and
finite-state is called an \textit{automata group}.

\subsection{\textbf{Virtual endomorphisms.}}

Given a subgroup $H$ of $G$ of finite index $m$, a homomorphism $%
f:H\rightarrow G$ is called a \textit{virtual endomorphism} of $G$. A
subgroup $K$ of $H$ is $f$\textit{-invariant} provided $K^{f}\subseteq K$.
The maximal subgroup of $H$ which is both $f$-invariant and normal in $G$ is
called the $f$\textit{-core} of $H$; if this subgroup is trivial then $\ f$
is said to be \textit{simple}. If $G$ is a self-similar subgroup of $%
\mathcal{A}_{m}$, then for $H=Fix_{G}(0)$, the subgroup stabilizer of the
vertex $0\in Y$, we have the projection $f:H\rightarrow G$ which can be seen
to be simple.

Let $G$ be a group, $H$ a subgroup of $G$ of finite index $m$ with right
transversal $T=\{t_{0},...,t_{m-1}\}$ in $G$. Then the induced permutation
representation $\sigma $ of $G$ on $T$ is transitive and we define the
Scheier function $\theta :G\times T\rightarrow H$ by $\theta \left(
g,t_{i}\right) =t_{i}g(t_{j})^{-1}$ where $Ht_{j}=Ht_{i}g$.

\subsection{Recursive \textbf{Kaloujnine-Krasner}}

The Kaloujnine-Krasner Theorem \cite{KK} provides us with an extension of $\sigma 
$ to a representation $\varphi :$ $G\rightarrow H\wr G^{\sigma }$ defined by%
\begin{equation*}
\varphi :g\mapsto (\theta \left( g,t_{0}\right) ,...,\theta \left(
g,t_{m-1}\right) )g^{\sigma }\text{.}
\end{equation*}%
On applying\ the virtual endomorphism $f$ to $\theta \left( g,t_{i}\right) $%
, we obtain $\theta \left( g,t_{i}\right) ^{f}$ $\in G$ which allows us to
repeat the above representation $\varphi $ to this element. Thus the
Kaloujnine-Krasner representation extends recursively to a representation on
the $m$-tree $\mathcal{T}_{m}$, , indicated by the same symbol, as follows

\begin{equation*}
\varphi :\text{ }g\mapsto \left( \theta \left( g,t_{0}\right) ^{f\varphi
},...,\theta \left( g,t_{m-1}\right) ^{f\varphi }\right) \,g^{\sigma }.
\end{equation*}%
The kernel of the representation $\varphi $ is the $f$\textit{-core} of $H$
\cite{NS}.

\ Given a group $G$, $m\geq 1$ and $G$-data $\left( \mathbf{m},\mathbf{H,F}%
\right) $ we re-work the recursive representation in the transitive case as
follows. For each $H_{i}$ choose a right transversal $T_{i}=\{t_{i1},...,t_{im_i}\}$ and let $\theta
_{i}$ be the corresponding Schreier function . Then define $\varphi
:G\rightarrow \mathcal{A}_{m}$ by%
\begin{equation*}
\varphi :\text{ }g\mapsto \left( \theta _{i}\left( g,t\right) ^{f_{i}\varphi
}\mid 1\leq i\leq s,t\in T_{i}\right) \,g^{\sigma }.
\end{equation*}

\section{Proof of Proposition A}

We reformulate the proposition more concretely as follows

\begin{propA}
\label{TA}

 Given a group $G$, $m\geq 1$ and $G$-data $\left( \mathbf{m},%
\mathbf{H,F}\right) $, the function $\varphi :G\rightarrow \mathcal{A}_{m}$
defined by 
\begin{equation*}
\text{ }g\mapsto \left( \theta _{i}\left( g,t\right) ^{f_{i}\varphi }\mid
1\leq i\leq s, \, t\in T_{i}\right) \,g^{\sigma }.
\end{equation*}%
is a state-closed representation of $G$ on the $m$-tree with kernel the $%
\mathbf{F}$-core of $\mathbf{H}$.
  
\end{propA}

\begin{proof}
Consider $g$ and $h$ elements of $G$. Clearly $(gh)^{\sigma}=g^{\sigma}h^{%
\sigma}$, so if $i\in Y$, then 
\begin{equation*}
i^{(gh)^{\varphi}}
=i^{(gh)^{\sigma}}=i^{g^{\sigma}h^{\sigma}}=i^{g^{\varphi}h^{\varphi}}.
\end{equation*}
By induction on the length of word in $Y$, we have that

\begin{equation*}
(iu)^{(gh)^{\varphi}}=i^{(gh)^{\sigma}}u^{(gh)^{\varphi_i}}=i^{(gh)^{%
\sigma}}u^{(g^{\varphi}h^{\varphi})_i}=(iu)^{g^{\varphi}h^{\varphi}}.
\end{equation*}

\noindent Therefore $\varphi$ is a homomorphism. By construction, $G^{\varphi}$ is
state-closed.

Let $K$ be the $\textbf{F}$-core of $\textbf{H}$ and $x\in K$. Then $%
x^{\sigma}=1$; in fact, $k^{x^{\sigma}}=l$ if and only if $%
H_it_{ik}x=H_it_{il}$ for some $i=1,...,s$. But $%
H_it_{ik}x=H_ix^{t_{ik}^{-1}}t_{ik}=Ht_{il}$, so $k=l$ and $x^{\sigma}=1$.
Since $K^{f_{i}}\leqslant K$, ${(\theta_i(x,t)^{f_i\varphi})}^\sigma$ is a trivial
permutation for all $i=1,...,s$ and $t \in T_i$. So $x \in \ker \varphi$.
Clearly $\ker \varphi\leqslant \cap _{i=1}^{s}H_{i}$, $\ker \varphi
\vartriangleleft G$ and $\ker \varphi ^{f_{i}}\leqslant \ker \varphi,\forall
i=1,...,s$.
\end{proof}

\begin{prop} \label{r}
The group $\mathbb{Z}\wr \mathbb{Z}$ is a $3$-state and $3$-letter automata.
\end{prop}

\begin{proof}
First we prove that $\mathbb{Z}\wr \mathbb{Z}$ has a faithful state-closed representation of degree 4 which we then reduced to degree 3.

By Brunner and Sidki \cite[Theorem 1]{BS1}, if $H$ is an abelian subgroup of 
$\mathcal{A}_{2}$ and $\alpha =((e,e),(\alpha ,e))(0\,1)$, then $\langle 
\tilde{H},\alpha \rangle \simeq \tilde{H}\wr \langle \alpha \rangle $, where 
$\tilde{H}=\{\tilde{h}=((\tilde{h},h),(e,e))\mid h\in H\}.$ Note that $%
\tilde{H}\simeq H$. If $H=\langle \alpha \rangle $, then $\tilde{H}\wr
\langle \alpha \rangle =\langle \tilde{\alpha}\rangle \wr \langle \alpha
\rangle \simeq \mathbb{Z}\wr \mathbb{Z}$. Since 
\begin{equation*}
\langle \tilde{\alpha}\rangle \wr \langle \alpha \rangle =\langle \tilde{%
\alpha}=((\tilde{\alpha},\alpha ),(e,e)),\alpha =((e,e),(\alpha
,e))(0\,1)\rangle ,
\end{equation*}%
it follows that $\langle \tilde{\alpha}\rangle \wr \langle \alpha \rangle $
is level two, state-closed group. 

%Let $\bar{\sigma}(\alpha )=((\sigma _{0}(\alpha ),\sigma _{1}(\alpha ))\sigma (\alpha )$ be an automorphism and re-write the indexes as $00:=\mathbf{0}$, $01:=\mathbf{1%}$, $10:=\mathbf{2}$, and $11:=\mathbf{3}$. Then we note that  $\bar{\sigma}(\alpha )$ can be seen as a permutation $\sigma (\alpha )$ of the set $\{\mathbf{0},\mathbf{% 1},\mathbf{2},\mathbf{3}\}$ and each $\alpha _{ij}$ can be seen as an automorphism $\alpha _{2i+j}$ of the tree $\mathcal{T}_{4}$. Hence thehomomorphism

Thus, by $2$-inflation, $\Psi: G \rightarrow \mathcal{A}_4$
%\begin{equation*} \Psi :G\rightarrow \mathcal{A}_{4} \end{equation*}% \begin{equation*} \,\,\,\,\,\,\,\,\,\,\,\,\,\,\,\,\,\,\,\,\,\,\,\,\,\,\,\,\,\,\,\,\,\,\,\,\,\,% \,\,\,\,\,\,\,\,\,\,\,\,\,\, \,\alpha \mapsto (\alpha _{\mathbf{0}},\alpha _{% \mathbf{1}},\alpha _{\mathbf{2}},\alpha _{\mathbf{3}})\bar{\sigma}(\alpha ) \end{equation*}% So,
\begin{equation*}
(\langle \tilde{\alpha}\rangle \wr \langle \alpha \rangle )^{\Psi }=\langle 
\tilde{\alpha}^{\Psi }=(\tilde{\alpha}^{\Psi },\alpha ^{\Psi },e,e),\alpha
^{\Psi }=(e,e,\alpha ^{\Psi },e)(0,2)(1,3)\rangle \simeq \mathbb{Z}\wr 
\mathbb{Z}
\end{equation*}%
is a faithful state-closed representation of $\mathbb{Z}\wr \mathbb{Z}$ of
degree $4$.

Now note that the map 
\begin{equation*}
\alpha ^{\Psi }\mapsto \alpha _{1}=(e,\alpha _{1},e)(0,1)
\end{equation*}%
\begin{equation*}
\tilde{\alpha}^{\Psi }\mapsto \beta =(\beta ,e,\alpha _{1})
\end{equation*}%
extends to an isomorphism $\phi $ from 
\begin{equation*}
\langle \tilde{\alpha}^{\Psi }=(\tilde{\alpha}^{\Psi },\alpha ^{\Psi
},e,e),\alpha ^{\Psi }=(e,e,\alpha ^{\Psi },e)(0,2)(1,3)\rangle 
\end{equation*}%
to 
\begin{equation*}
\langle \beta =(\beta ,e,\alpha _{1}),\alpha _{1}=(e,\alpha
_{1},e)(0,1)\rangle .
\end{equation*}%
Therefore $\mathbb{Z}\wr \mathbb{Z}$ is a $3$-state and $3$-letter automata, as in Diagram 1.
\end{proof}

We prove in Section 5 a more general form of this proposition.

\section{Proof of Theorem B}

% Theorem B is a source for constructing nontrivial examples of finite-state, state-closed groups.

\begin{thmB}
\label{TB} Let $G$ be a self-similar group of degree $m$ and orbit-type $%
\left( m_{1},..,m_{s}\right) $. Then  the following hold.

\begin{itemize}
\item[1)] $G^{(\omega )}$ admits a faithful state-closed representation of
degree $m+1$ and orbit-type $\left( m_{1},..,m_{s},1\right)$; in particular,
for $G=\mathbb{Z}$, the representation of the group $\mathbb{Z}^{(\omega )}$
is of orbit-type $\left( 2,1\right),$ and is in addition finite-state.

\item[2)] Let $K$ be a regular subgroup of $Sym(\{1,...,s\})$. Then the
restricted wreath product $G\wr K$ admits a faithful
state-closed representation of degree $(m_1\cdot m_2\cdot...\cdot m_s) \cdot s$; in particular, the group $(\mathbb{Z} \wr \mathbb{Z})
\wr C_{2}$ is finite-state and self-similar of degree $4$.
\end{itemize}
\end{thmB}

\begin{proof}
Let $G$ be a state-closed group. By Proposition A, there is data $(\mathbf{m}%
,\mathbf{H,F})$ such that $\mathbf{H}$ is $\mathbf{F}$ core-free.

(1) The subgroup $L_{i} = \{(h, g_{2}, g_{3},...) \in G^{(\omega)}
\mid h \in H_{i} \}$ has index $m_{i}$ in $G^{(\omega)}$. For each $i = 1,
..., s$ define the homomorphism $\bar{f}_{i} : L_{i} \rightarrow
G^{(\omega)} $ by 
\begin{equation*}
(h, g_{2}, g_{3},...)^{\bar{f}_{i}} = (h^{f_{i}}, g_{2}, g_{3},...)
\end{equation*}
and the homomorphism $\bar{f}_{s + 1}: L_{s+1} = G^{(\omega)} \rightarrow
G^{(\omega)}$ by 
\begin{equation*}
(g_{1}, g_{2}, ...)^{\bar{f}_{s + 1}} = (g_{2}, g_{3}, ...).
\end{equation*}
It is clear that 
\begin{equation*}
\langle L \leqslant \cap_{i = 1}^{s+1} L_{i} \mid L \vartriangleleft
G^{\omega}, L^{\bar{f}_{i}} \leqslant L, \forall i = 1, ..., s + 1 \rangle
\end{equation*}
is trivial. By Proposition A, the group $G^{(\omega)}$ is state-closed with
 respect to the data $(\mathbf{\bar{m},\bar{H},\bar{F}})$, where $\mathbf{\bar{m%
}} = (m_{1},...,m_{s},1)$, $\mathbf{\bar{H}=}\{L_1,...,L_s, L_{s+1}\}$ and $%
\mathbf{\bar{F}} = \{\bar{f}_{1},...,\bar{f}_{s},\bar{f}_{s+1}\}.$

Note that $f:2\mathbb{Z}%
\rightarrow \mathbb{Z}$  defined by $2n\mapsto n$ is a simple virtual endomorphism and therefore $\mathbb{Z}$ is self-similar of degree $2$.
Thus the group $\mathbb{Z}^{(\omega )}$ is self-similar  with respect to the data $((2, \, 1), \{L_1, L_2\}, \{f_1, f_2\})$ where $f_{1}:(2n_{1},n_{2},...)\mapsto (n_{1},n_{2},...)$ and $%
f_{2}:(n_{1},n_{2},...)\mapsto (n_{2},n_{3},...)$. On defining the transversals $T_{1}=\{e,(1,0,0,...)%
\}$ and $T_{2}=\{e\}$ of $L_{1}$ and $L_{2}$ respectively,  we obtain the following representation of $\mathbb{Z}^{(\omega)}$ 
\begin{equation*}
\langle \alpha _{1}=(e,\alpha _{1},e)(0\,1),\newline
\alpha _{i}=(\alpha _{i},\alpha _{i},\alpha _{i-1})\mid \newline
i=2,3,4,...\rangle
\end{equation*}%
which is faithful and finite-state.
\begin{center}
\begin{tikzpicture}[shorten >=3pt,node distance=2.5cm,on grid,auto] 

  \node[state] (e) [] {$e$};
  \node[state] (a) [right of=e] {$\alpha_1$};
  \node[state] (b) [right of=a] {$\alpha_2$};
  \node[state] (c) [right of=b] {$\alpha_3$};
  \node (d) [right of=c] {$...$};

  \path [->] (e) edge  [loop above]    node {$0|0, 1|1, 2|2$} (e)
        (a) edge              node {$0|1,2|2$} (e)
        (b) edge              node {$2|2$} (a)
        (c) edge              node {$2|2$} (b)
        (d) edge              node {$2|2$} (c)
        (a) edge         [loop above]     node {$1|0$} (a)
        (b) edge      [loop above]        node {$0|0,1|1$} (b)
        (c) edge      [loop above]        node {$0|0,1|1$} (c);

\end{tikzpicture}

Diagram 2
\end{center}

(2) Let $l$ be an integer such that $H_{i} \neq G$ for $1 \leq i \leq l$ and 
$H_{i} = G$ for $l + 1 \leq i \leq s$. Define $H = H_{1} \times ... \times
H_{s}$ and $W = G \wr K$. Then $[W : H] = s(m_{1} ... m_{l})$ and the
endomorphism $f: H \rightarrow W$ given by 
\begin{equation*}
(h_{1}, ..., h_{s}) \mapsto (h_{1}^{f_{1}}, ..., h_{s}^{f_{s}})
\end{equation*}
is well-defined. Let $L$ be a subgroup of $H$, normal in $W$, and $f$%
-invariant, and let $g = (g_{1}, ..., g_{s}) \in L$. Since $K$ is a
transitive group of degree $s$ follows that for each $1 \leq i \neq j \leq s$
there exists $h \in K$ such that $(i)h = j$. But 
\begin{equation*}
g^{h f} = \left(g_{(1)h}, ..., g_{(s)h}\right)^{f} = \left(g_{(1)h}^{f_{1}},
..., g_{(s)h}^{f_{s}}\right) \in L.
\end{equation*}
Thus $g_{r} \in \langle K \leqslant \cap_{i = 1}^{s} H_{i} \mid K
\vartriangleleft G, K^{f_{i}} \leqslant K, \forall i = 1, ..., s \rangle =
\{1\}$ for each $r = 1, ..., s$. Therefore $L = \{1\}$ and $G \wr K$ is self
-similar of degree $s.(m_{1} ... m_{l})$.

On applying Proposition \ref{r} and Proposition A we obtain: 
\begin{equation*}
(\mathbb{Z}\wr \mathbb{Z})\wr C_{2}\simeq \langle \sigma
=(0,2)(1,3),\gamma =(\gamma ,e,\alpha ^{\sigma },\alpha ^{\sigma }),\alpha
=(e,\alpha ,e,e)(0,1)\rangle;
\end{equation*}%
that is, the group $(\mathbb{Z}\wr \mathbb{Z})\wr C_{2}$ is generated by the following $%
5 $-state and $4$-letter automaton:

\begin{center}
    \begin{tikzpicture}[shorten >=3pt,node distance=3.7cm,on grid,auto] 

  \node[state] (e) [] {$e$};
  \node[state] (a) [below left of=e] {$\alpha$};
  \node[state] (g) [below right of=e] {$\gamma$};
  \node[state] (p) [below right of=a] {$\alpha^{\sigma}$};
  \node[state] (s) [left of=e] {$\sigma$};

  \path [->] (e) edge        [loop above]       node {$0|0, 1|1, 2|2, 3|3$} (e)
        (s) edge              node {$0|2,1|3,2|0,3|1$} (e)
        (a) edge              node {$0|1,2|2,3|3$} (e)
        (a) edge        [loop left]       node {$1|0$} (a)
        (g) edge          [loop right]     node {$0|0$} (g)
        (g) edge              node [swap]  {$2|2,3|3$} (e)
        (p) edge              node {$3|2$} (a)
        (g) edge              node  {$1|1$} (p)
        (p) edge              node [swap] {$0|0, 1|1, 2|3$} (e);

\end{tikzpicture}

Diagram 3
\end{center}

\end{proof}

\section{Proof of Theorem C}

First we will prove the second case of Theorem C.

\begin{prop}
\label{P3.1} Let $G=\mathbb{Z}^{l}\wr \mathbb{Z}^{d}$. Then $G$ is an
automata group of degree $4$. In case $d=1$, the degree is $3$.
\end{prop}

\begin{proof}
Denote $\mathbb{Z}^{l}$ by $A$ and $\mathbb{Z}^{d}$ by $X$. \ Then, the normal closure of $A$ in $G$ is $A^X$ and  we have
a semi-direct product form for the group $G=A^{X}\mathbb{\cdot }X$ .

Define the subgroups%
\begin{eqnarray*}
H_{1} &=&A^{X}\langle
x_{1}^{2},x_{2},...,x_{d}\rangle , \\
H_{2} &=&H_{3}=G,\text{ }\bigcap_{i=1}^{3}H_{i}=H_{1}\text{.}
\end{eqnarray*}%
Also, for $i=1,2,3$, define the homomorphisms $f_{i}:H_{i}\rightarrow G$
which extend the maps:%
\begin{eqnarray*}
f_{1} &:&\,x_{1}^{2}\mapsto x_{1},\,x_{i}\mapsto x_{i}\text{ }\left( 2\leq
i\leq d\right) , \\
a_{i}^{x_{1}^{2n}q\left( x_{2},x_{3}...,x_{d}\right) } &\mapsto
&a_{i-1}^{x_{1}^{n}q\left( x_{2},x_{3}...,x_{d}\right) },\text{ }%
a_{i}^{x_{1}^{2n+1}q\left( x_{2},x_{3}...,x_{d}\right) }\mapsto e\text{ };
\end{eqnarray*}

% $$f_2: y_{1} \mapsto y_{l} \mapsto y_{l-1} \mapsto ... \mapsto y_{2} \mapsto y_{1}, \, x_{1}  \mapsto  x_{1}, ..., x_{d} \mapsto x_{d};$$

\begin{eqnarray*}
f_{2} &:&\,x_{i}\mapsto x_{i-1}\text{ }\left( 1\leq i\leq d\right) , \\
a_{i}^{p\left( x_{1},x_{2}...,x_{d}\right) } &\mapsto &a_{i}^{p\left(
x_{d},x_{1}...,x_{d-1}\right) }\text{ }\left( 1\leq i\leq l\right) ;
\end{eqnarray*}

and%
\begin{eqnarray*}
f_{3} &:&X\mapsto \left\{ e\right\} \text{ }, \\
a_{1}^{p\left( x_{1},x_{2}...,x_{d}\right) } &\mapsto &x_{1}^{p\left(
1,1...,1\right) },\text{ } \\
a_{i}^{p\left( x_{1},x_{2}...,x_{d}\right) } &\mapsto &e\text{ }\left( 2\leq
i\leq l\right)
\end{eqnarray*}

Then, $A^{ X }\langle
x_{2},...,x_{d}\rangle $ is the $f_{1}$-core of $H_{1}$ and both $H_{2}$, $%
H_{3}$ are their own $f$-cores.

Let $K$ be the $\mathbf{F}$\textbf{-}core of $\mathbf{H}$. Then $K\leq A^{ X }\langle x_{2},...,x_{d}\rangle $ and
by applying $f_{2}$ we find $K\leq A^{ X} $.

Suppose $K$ is non-trivial. Then, as $K$ is normal in $G$ it follows that $%
K_{+}=K\cap A^{\mathbb{Z}\left[ X\right] }$ is non-trivial. Elements $h$ of $%
K_{+}$ have an unique form 
\begin{equation*}
h=a_{1}^{p_{1}}a_{2}^{p_{2}}...a_{l}^{p_{l}}
\end{equation*}%
where $p_{i}=p_{i}(x_{1},x_{2}...,x_{d})$ $\in \mathbb{Z}\left[ X\right] $.
Define $\delta _{x_{j}}(h) $ to be the maximum $x_{j}$-degree of $p_{i}$ for
all $i$. If the maximum of all $\delta _{x_{j}}(h)$ occurs for $j=k$ then by
applying an adequate power of $f_{2}$ to $h$, we may assume $j=1$.

Choose an $h\in K_{+}$ such that $h\not=e$ and which involves a minimum
number of variables from $\left\{ x_{1},x_{2}...,x_{d}\right\}$. \newline

\noindent (1) Suppose $p_{i}$ is constant for all $i$. There exists $j$ such
that $p_{j}\not=0$. Then on applying $f_{3}$ to $h$ we get $%
x_{1}^{p_{1}}\not=e\in K$ which is impossible. \newline

\noindent (2) Write $\delta _{x_{1}}\left( h\right) =n$ then $n\not=0$. 
\newline

\noindent (2.1) Suppose $n=2k$. Then on applying $f_{1}$ to $h$ we obtain $%
h^{\prime }\not=e$ and $\delta _{x_{1}}\left( h^{\prime }\right) =k$ which
is absurd. \newline

\noindent (2.2) Suppose $\delta _{x_{1}}\left( h\right) =2k+1$. Then
conjugate by $x_{1}$ to get%
\begin{eqnarray*}
h^{\prime } &=&h^{x_{1}}=a_{1}^{p_{1}^{\prime }}a_{2}^{p_{2}^{\prime
}}...a_{l}^{p_{l}^{\prime }}\in K_{+}, \\
p_{i}^{\prime } &=&p_{i}x_{1},\text{ }\delta _{x_{1}}\left( h^{\prime
}\right) =2k+2\text{.}
\end{eqnarray*}%
On applying $f_{1}$ to $h^{\prime }$ we get $h^{\prime \prime }\in K_{+}$
with $\delta _{x_{1}}\left( h^{\prime \prime }\right) =k+1$. Now, $k+1<2k+1$%
, unless $k=0$; that is, we have $\delta _{x_{1}}\left( h\right) =1$ and for
all $i$, 
\begin{equation*}
p_{i}=p_{i0}+p_{i1}x_{1}
\end{equation*}%
where $p_{i0},p_{i1}\in $ $\mathbb{Z}\left[ x_{2}...,x_{d}\right] $. Then,
as before, for $h^{\prime }=h^{x_{1}}$ we have $p_{i}^{\prime
}=p_{i0}x_{1}+p_{i1}x_{1}^{2}$ and $h^{\prime \prime }=\left(
a_{1}^{p_{11}}a_{2}^{p_{21}}...a_{l}^{p_{l1}}\right) ^{x_{1}}$ $\in
K_{+}\backslash \left\{ e\right\} $. Since 
\begin{equation*}
h^{\prime \prime \prime }=\left( h^{\prime \prime }\right)
^{x_{1}^{-1}}=a_{1}^{p_{11}}a_{2}^{p_{21}}...a_{l}^{p_{l1}}\in
K_{+}\backslash \left\{ e\right\} \text{ }
\end{equation*}%
which involves a lesser number of variables than $h$, we have a
contradiction.

With notation of Proposition \ref{P3.1}, a faithful self-similar
representation of $G=\mathbb{Z}^{l}\wr \mathbb{Z}^{d}$ with respect the data%
\begin{equation*}
((2,1,1),\{H_{1},\text{ }H_{2} = G, H_{3}=G\},\{f_{1},f_{2},f_{3}\})
\end{equation*}%
and the transversals $T_{i}$ of $H_{i}$ in $G$ defined by 
\begin{equation*}
T_{1}=\{e,x_{1}\},T_{2}=T_{3}=\{e\}\text{,}
\end{equation*}%
is 
\begin{equation*}
G^{\varphi }=\langle \gamma _{1},...,\gamma _{l}\rangle \wr \langle \alpha
_{1},...,\alpha _{d}\rangle 
\end{equation*}%
where
\begin{equation*}
\gamma _{1}=(\gamma _{l},e,\gamma _{1},\alpha _{1}),\,\gamma _{2}=(\gamma
_{1},e,\gamma _{2},e),...,\gamma _{l}=(\gamma _{l-1},e,\gamma _{l},e),
\end{equation*}%
\begin{equation*}
\alpha _{1}=(e,\alpha _{1},\alpha _{d},e)(0\,1),\,\alpha _{2}=(\alpha
_{2},\alpha _{2},\alpha _{1},e),...,\alpha _{d}=(\alpha _{d},\alpha
_{d},\alpha _{d-1},e).
\end{equation*}%
Therefore, $G^{\varphi }$ is finitely generated, finite-state and
self-similar; that is, $G~$is an automata group. We note that if $d=1$ then $%
H_{2}$ and $f_{2}$ are superfluous.
\end{proof}

\subsection{\noindent \textbf{General concatenation}. }

Let $G_{1}=A_{1}\wr U$, $G_{2}=A_{2}\wr U$ and $G=\left( A_{1}\oplus
A_{2}\right) \wr U$. For $i=1,2$, define the $G_{i}$-data $\left( \mathbf{%
m_{i},H_{i},F_{i}}\right) $ , where $\mathbf{m_{i}=}\left(
m_{i1},...,m_{is_{i}}\right) $, $\mathbf{H_{i}}=\{H_{i1},...,H_{is_{i}}\}$
and $\mathbf{F_{i}}=\{f_{i1},...,f_{is_{i}}\}$. Furthermore, define the data 
$G$-data $\left( \mathbf{m,H,F}\right) $ , where $\mathbf{m}$ is the
concatenation $\left( \mathbf{m_{1},m_{2}}\right) $, 
\begin{eqnarray*}
\mathbf{H} &\mathbf{=}&\left\{ \tilde{H}_{1j}=\left( {A_{2}}^{U}\right)
\cdot H_{1j}\text{ }\left( 1\leq j\leq s_{1}\right) \right\}  \\
&&\cup \left\{ \tilde{H}_{2k}=\left( {A_{1}}^{U}\right) \cdot H_{2k}\text{ }%
\left( 1\leq k\leq s_{2}\right) \right\} \text{,}
\end{eqnarray*}%
\begin{equation*}
\mathbf{F=}\{\tilde{f}_{11},...,\tilde{f}_{1s_{1}},\tilde{f}_{21},...,\tilde{%
f}_{2s_{2}}\}
\end{equation*}%
where $\tilde{f}_{1j}:\tilde{H}_{1j}\rightarrow G, \,\, 1\leq j \leq s_1,$ is defined by 
\begin{equation*}
\tilde{f}_{1j}:ah\mapsto h^{f_{1j}}\text{, for }a\in {A_{2}}^{U},h\in H_{1j}
\end{equation*}%
and $\tilde{f}_{2k}:\tilde{H}_{2k}\rightarrow G, \,\, 1\leq k \leq s_2,$ is defined by 
\begin{equation*}
\tilde{f}_{2k}:ah\mapsto h^{f_{2k}}\text{, for }a\in {A_{1}}^{U}\text{, }%
h\in H_{2k}\text{.}
\end{equation*}

For $i=1,2$, let $G_{i}$ has its state-closed representation with respect
to $\left( \mathbf{m_{i},H_{i},F_{i}}\right) $ with  $\mathbf{F_{i}}$ core $%
K_{i}$. Likewise, let $G$ has its state-closed representation with respect
to $\left( \mathbf{m,H,F}\right) $ with $\mathbf{F}$-core $K$.

\begin{prop}
\label{P3.2} Maintaining the above notation:

\begin{enumerate}
\item 
\begin{equation*}
K\cap \left( A_{1}\oplus A_{2}\right) ^{U}=K_{1}\left( A_{2}\right) ^{U}\cap
K_{2}\left( A_{1}\right) ^{U}\text{;}
\end{equation*}

\item Suppose the above state-closed representations of $G_{1}$ and $G_{2}$
are faithful. Then so is the corresponding state-closed representation of $G$;

\item If $G_{1}$ and $G_{2}$ are finite-state then $G$ is also finite-state.
\end{enumerate}
\end{prop}

\begin{proof}
Define for $i=1,2$,%
\begin{equation*}
R_{i}=\langle S\leqslant \cap _{j=1}^{s_{i}}H_{ij}\mid S\vartriangleleft G,%
\text{ }S^{\widetilde{f}_{ij}}\leqslant S,\forall j=1,...,s_{i}\rangle ,
\end{equation*}

\begin{enumerate}
\item We have 
\begin{eqnarray*}
K &=&R_{1}\cap R_{2}, \\
R_{1}\cap \left( A_{1}\oplus A_{2}\right) ^{U} &=&K_{1}\left( A_{2}\right)
^{U}, \\
R_{2}\cap \left( A_{1}\oplus A_{2}\right) ^{U} &=&K_{2}\left( A_{1}\right)
^{U}, \\
K\cap \left( A_{1}\oplus A_{2}\right) ^{U} &=&K_{1}\left( A_{2}\right)
^{U}\cap K_{2}\left( A_{1}\right) ^{U}\text{.}
\end{eqnarray*}

\item As $K_{1}=K_{2}=\left\{ e\right\} $, we find%
\begin{equation*}
K\cap \left( A_{1}\oplus A_{2}\right) ^{U}=\left( A_{2}\right) ^{U}\cap
\left( A_{1}\right) ^{U}=\left\{ e\right\} \text{.}
\end{equation*}%
Now since both $K$ and $\left( A_{1}\oplus A_{2}\right) ^{U}$ are normal
subgroups of $G$, it follows that $\,K$ centralizes $\left( A_{1}\oplus
A_{2}\right) ^{U}$. However $\left( A_{1}\oplus A_{2}\right) ^{U}$ contains
its own centralizer in $G$. Thus, $K=K\cap \left( A_{1}\oplus A_{2}\right)
^{U}$ which is trivial by (1). \\ 

\item There exist transversals of $H_{11}, ..., H_{1s_{1}}$ in $G_{1}$ and of $H_{21}, ...,$ $H_{2s_{2}}$ in $G_{2}$ such that $G_{1}$ and $G_{2}$ are finite-state. These transversals also induce a finite-state representation of $G$. 
\end{enumerate}
\end{proof}

Now Theorem C follows directly from the above proposition
and Proposition \ref{P3.1}. \\

\noindent \textit{\textbf{Question 1}. Is there a faithful state-closed
representation of degree }$3$\textit{\ for the group }$\mathbb{Z}^{l}\wr 
\mathbb{Z}^{d}$\textit{\ when }$l, d\geq 2$\textit{?}

\section{Proof of Theorem D}

We start with the following observation. Let $k$ be a field and $X = \langle x, y \rangle \simeq \mathbb{Z}^{2}$. Consider the following equivalence relation on the $k$-algebra $k[X]$ defined by: 
\begin{equation*}
p(x,y)\sim q(x,y)\text{ iff }q(x,y)= u p(x,y),\text{ where }u\text{ is a unit of }k[X].
\end{equation*}%
 Let $p'(x,y)= x^s.y^t.p(x,y)$, where $s,t \geq 0$ minimal such that $p'(x,y)$ in $k[x,y]$. Let $m, n$ be respectively the $x$-degree and $y$-degree of $p'(x,y)$ and define $\delta(p(x,y)) = (m, n)$. Then, for $p(x,y), q(x,y)$ non-invertible elements of $k[X]$, $\delta(p(x,y)q(x,y)) \geq \delta( p(x,y)), \delta(q(x,y))$.  Let $p_i(x,y)$ be a sequence of non-invertible elements of $k[X]$ where $i \geq 0$ and let $\delta( p_i(x,y)) = (m_i, n_i)$. Suppose $m_i, n_i \rightarrow \infty$ as $i \rightarrow \infty$. Then, the ideal $\bigcap_{i=0}^{\infty }\langle
p_{i}(x,y)\rangle $ is null.

\begin{thmD}
\label{TC copy(1)} Let $p$ a prime number then $C_{p}\wr \mathbb{Z}^{2}$ is
a self-similar group of degree $p+1$ of orbit-type $%
\left( p,1\right)$. Indeed, $C_{p}\wr \mathbb{Z}^{2}$ is generated by $\alpha = (\alpha, \alpha \sigma, ..., \alpha \sigma^{p-1}, \alpha \beta)$, $\sigma = (e, ..., e, \sigma)(0 \, 1 \, ... p-1)$ and $\beta = (e, ..., e, \alpha)$. In particular, the group $C_{2}\wr \mathbb{Z}^{2}$ is
self-similar of degree $3$.
\end{thmD}

\begin{proof}
Let $C_{p} = \langle a \rangle$, $\mathbb{Z}^{2} = \langle x, y \rangle$ and 
$G = C_{p} \wr \mathbb{Z}^{2}$. Define the subgroups $H_{1} = G^{\prime
}\langle x, y \rangle$ and $H_{2} = G$. Note that $[G : H_{1}] = p$.
Elements of  $H_{1}$  have the unique form $a^{s(x,y)}\cdot x^i y^j$ where $s(x,y)$ is an element of the ideal $\mathcal{I}$ of $\mathbb{Z} \langle x, y \rangle$, generated by $x-1$ and $y-1$. 
Also, elements of $\mathcal{I}$ have the unique form 
$$s(x,y)=p(x)(x-1)+q(y)(y-1)+r(x,y)(x-1)(y-1).$$
Define $f_{1} : H_{1} \rightarrow G$  by  $a^{s(x,y)} x^i y^j \mapsto a^{q(y)} y^j.$
Furthermore define $f_{2} : H_{2} \rightarrow G$ by  $a^{r(x,y)}x^i y^j \mapsto a^{r(y,xy)}x^j y^{i+j}.$
It can be checked directly that $f_{1}$ and $f_{2}$ are homomorphisms.

Suppose by contradiction that $K$ is a non-trivial subgroup of  $H_{1} \cap H_{2} = H_{1}$%
, normal in $G$ and $\{f_{1}, f_{2}\}$-invariant. Note that the subgroup $L
= K \cap G^{\prime }$ is trivial if and only if $K$ is trivial. Let $g = a^{s(x, y)}$ be a nontrivial element in $L$. So $a^{s(x, y)f_{1}} = a^{q(y)}$ and successive applications of $f_{1}$ in $a^{q(y)}$ results $q(y) = 0$, thus $x - 1$ divides $s(x, y)$.

Since $a^{s(x,y)f_{2}} = a^{s(y, xy)} \in L$, it follows that $x - 1$ also
divides $s(y, xy)$, this is, $s(y, xy) = (x - 1)t(x, y)$. Then 
\begin{equation*}
\begin{split}
a^{s(x, y)} &= a^{(s(y, xy))f_{2}^{-1}} \\
&= a^{((x-1)t_{1}(x, y))f_{2}^{-1}} \\
&= a^{(x^{-1}y-1)t_{2}(x^{-1}y, x)} \\
&= a^{x^{-1}(y - x)t_{2}(x^{-1}y, x)} \\
&= a^{(x - y)t_{3}(x, y)}
\end{split}
\end{equation*}
and $x - y$ divides $s(x, y)$. Iterating this argument, we find that $x^{n_{i}} - y^{n_{i - 1}}$ divides $s(x, y)$ for all $n_{i}$, where $n_i$ is the Fibonacci sequence defined by $n_{i} = n_{i - 1} + n_{i -2}$ and $n_{0} =
0, n_{1} = 1, n_{2} = 1$, $i \geq 0$. Hence $s(x, y) \in \bigcap_{i =
0}^{\infty} \langle x^{n_{i}} - y^{n_{i - 1}} \rangle = \{0\}$; a contradiction. Therefore $G$ is state-closed of degree $p + 1$ and
orbit-type $(p, 1)$.

Choose the transversals $T_{1}=\{e, a, ..., a^{p-1} \}$ and $T_{2}=\{e\}$ of $H_{1}$  and $H_{2}$, respectively. So we have a
state-closed representation of $G$  generated by the automorphisms
\begin{equation*}
\sigma =(e,e, ..., e, \sigma )(0\,1),\alpha =(\alpha , \alpha
\sigma, ..., \alpha \sigma^{p-1}, \alpha \beta ),\beta =(e, e, ..., e, \alpha ).
\end{equation*}%
Note that $\alpha^{m} \beta^{n} = (\alpha^{m}, (\alpha\sigma)^{m}, ..., (\alpha \sigma^{p-1})^{m}, \alpha^{m + n} \beta^{m})$, for $m, n \in \mathbb{Z}$, hence $\{\alpha, \alpha \beta, \alpha^{2} \beta, \alpha^{3} \beta^{2}, ..., \alpha^{n_{i}} \beta^{n_{i - 1}}, ...\} \subset Q(\alpha)$ and this representation is not finite-state.
\end{proof}

\noindent \textit{\textbf{Question 2}. Is there a faithful finite-state and
state-closed representation for the group }$C_{2}\wr \mathbb{Z}^{2}$\textit{%
\ of degree }$3$\textit{?} \newline

\section{Proof of Theorem E}

\noindent Let $G$ be a state-closed group with respect the data $(\mathbf{m}, \mathbf{H%
}, \mathbf{F})$, where $\mathbf{m} = (m_{1}, ...,
m_{s})$, with $m_{1} \geq 2, ..., m_{s} \geq 2$. 
For each $i=1,...,s$ define $G_{i0}=G$, $G_{ij}= (G_{i (j - 1)})^{f_i^{-1}}$ ($j>0$) and $G_{\omega_i}= \cap_{j\geq 0}G_{ij}.$ With this notation we have:

\begin{thmE}
Let $B$ be a finite abelian group and $G$ be a self-similar group of
orbit-type $(m_{1}, ...,
m_{s})$, with $m_{1} \geq 2, ..., m_{s} \geq 2$. Then the group $$B^{((G_{\omega_1}\setminus G) \times ... \times (G_{\omega_s}\setminus G))}\rtimes G^s$$
is self-similar of orbit-type $(|B|.m_{1} ... m_{s}, \, 1)$.
\end{thmE}

\begin{proof}
Let $\mathcal{G}$ be the group $B^{((G_{\omega_1}\setminus G) \times ... \times (G_{\omega_s} \setminus G))}\rtimes G^s$ and let $H$ be the subgroup $\prod_{i = 1}^{s} H_{i}$ of $G^s$. Define the set of s-tuples of cosets
$$U=\prod _{i=1}^{s}(G_{\omega_i}\setminus G) $$
and define the group extension
\begin{equation*}
\mathcal{H}=\left\{ \phi : U  \rightarrow B\text{ finitely
supported, }\prod_{\bar{g} \in U}\phi (\bar{g})=1\right\} \rtimes H.
\end{equation*}%

Note that the index $[\mathcal{G}:\mathcal{H}]$ is $|B|\cdot m_1...m_s$. Now consider the map  defined on the s-tuples of cosets.
\begin{eqnarray*}
\lambda&:&  \prod_{i=1}^{s}(G_{\omega_i}\setminus H_{{i}}) \rightarrow \prod_{i=1}^{s}(G_{\omega_i}\setminus G)
\end{eqnarray*}%

\begin{equation*}
\, \, \, \, \, \, \, \, \, \, \, \, \, \, \, \, \, \, \, \, \, \, \, \, (G_{\omega_1}h_1,...,G_{\omega_s}h_s)\mapsto  (G_{\omega_1}h_1^{f_{1}},...,G_{\omega_s}h_s^{f_{s}}).
\end{equation*}%
If $(G_{\omega_i}h_i)^{\lambda}=(G_{\omega_i}h_i')^{\lambda}$, then $(h_i(h_i')^{-1})^{f_{i}}\in  G_{\omega_{i}}$ for each $1\leq i \leq s$; therefore $\lambda$ is injective.

 Define 
\begin{eqnarray*}
\lambda'&:&  \prod_{i=1}^{s}(G_{\omega_i}\setminus G) \rightarrow \prod_{i=1}^{s}(G_{\omega_i}\setminus G)
\end{eqnarray*}%
by

$$ \bar{y}=(G_{\omega_1}y_1,...,G_{\omega_s}y_s)\mapsto  \bar{x}=(G_{\omega_1}x_1,...,G_{\omega_s}x_s),$$ provided $\lambda(\bar{x})= \bar{y};$ otherwise, define it as $\bar{e}=(G_{\omega_1},...,G_{\omega_s}).$

Furthermore, define the map 
 $\chi_{1}: \mathcal{H} \rightarrow \mathcal{G}$
 \begin{equation*}
(\phi ,(h_{i})_{i = 1}^{s})\mapsto  \left(\bar{x} = (G_{\omega_1}x_{1}, ..., G_{\omega_s}x_{s})\mapsto \phi (\lambda'(\bar{x})) ,(h_{i}^{f_{i}})_{i = 1}^{s}\right).
\end{equation*}
Then it is direct to show that $\chi_{1}$ is a homomorphism.

Further, let $\chi_{2}: \mathcal{G} \rightarrow \mathcal{G}$ be
the homomorphism defined by
$$(\phi, (g_{1}, ..., g_{s})) \mapsto ((G_{\omega_i}x_{i})_{i = 1}^{s} \mapsto \phi((G_{\omega_{i}}x_{i+1})_{i = 1}^{s}), (g_{i+1})_{i = 1}^{s})).$$

We claim that the state-closed representation of $\mathcal{G}$ defined by the data $((|B|\prod_{i = 1}^{s} [G : \cap_{i = 1}^{s} H_{i}], 1), \{\mathcal{H}, \mathcal{G}\}, \{\chi_{1}, \chi_{2}\})$ is faithful.

Consider $K \leq \mathcal{H}$ with $K \vartriangleleft \mathcal{G}$ and $K^{\chi_{r}} \leq K$, $r = 1, 2$. Firstly, the assumption that ${\bf F}$ is core-free and the definition of $\chi_{2}$ imply that $K \leq B^{(R \setminus G)^{s}}$. Let $\phi
:U \rightarrow B$ be a non-trivial element in $K$, then $\prod_{\bar{g} \in U}\phi (\bar{g})=1$ and so $|Supp\,(\phi
)| \, \geq 2$. 
Choose $\phi $ with support $S$, of minimal cardinality. Since $%
\phi $ is non-trivial we can assume by $\mathcal{G}$-conjugation that $\phi (G_{\omega_1} , ...,G_{\omega_s})\neq 1$, and so $(G_{\omega_1} , ...,G_{\omega_s}) \in S $ and $S$ contains at least two elements.

Define $\phi_{j} = \phi^{\chi_{1}^{j}}$ for $j \geq 0$; it is given by $\phi_{j}(G_{\omega_1}a_1^{f_1}, ..., G_{\omega_s}a_s^{f_s}) = \phi_{(j - 1)}(G_{\omega_1}a_1, ..., G_{\omega_1}a_s)$ for all $(a_1,...,a_s) \in G_{1j}\times ...\times G_{sj}$, extended by the identity away from $( G_{1j}\times ...\times G_{sj})^{\lambda}$.
So 
$$\phi_1:G_{\omega_1} \setminus H_1^{f_1} \times ... \times G_{\omega_s} \setminus H_s^{f_s} \to B$$
$$\,\,\,\,\,\,\,\,\,\,\,\,\,\,\,\,\,\,\,\,\,\,\,\,\,\,\,\,\,\,\,\,\,\,\,\,\,\,\,\,\,\,\,\,\,\,\,\,\,\,\,\,\,\,\,\,\,\,\,\,\,\,\,\,\,\,\,\,\,\,\,\, \, \, (G_{\omega_1}a_1^{f_1}, ..., G_{\omega_s}a_s^{f_s})\mapsto \phi(G_{\omega_1}a_1, ..., G_{\omega_1}a_s) $$
and $$\,\,\,Supp(\phi_1) \,\,\,= \{(G_{\omega_1}a_1^{f_1}, ..., G_{\omega_s}a_s^{f_s}) \mid (G_{\omega_1}a_1, ..., G_{\omega_s}a_s) \in S  \}$$ $$=G_{\omega_1}(S^{\pi_1}\cap H_1)^{f_1}\times ... \times G_{\omega_s}(S^{\pi_s}\cap H_s)^{f_s},$$
where $\pi_i$, $i=1,...,s$, is the projection on the $i$-th coordinate. 

If 
$(G_{\omega_1}a_1^{f_1}, ..., G_{\omega_s}a_s^{f_s})\in Supp (\phi_1)$ then $(G_{\omega_1}a_1, ..., G_{\omega_s}a_s) \in S$ and so $|Supp (\phi_1)|\leq |S|$. By minimality of the cardinality of $S$ we have that $|Supp (\phi_1)|= |S|$.
Continuing in this manner, the support of $\phi_{j}$ is
$$G_{\omega_1}(S^{\pi_1}\cap G_{1j})^{f_1}\times ... \times G_{\omega_s}(S^{\pi_s}\cap G_{sj})^{f_s},$$
which has the same cardinality as $S$. Therefore, we have $S \subset (G_{1j} , ...,G_{sj})$ for all $j$ and thus $S \subset \{(G_{\omega_1} , ...,G_{\omega_s})\}$; a contradiction.

%So the support of $\phi_{1}$ is contained in $\prod_{k = 1}^{s}  G_{\omega_k} (\cap_{i = 1}^{s} (S^{\pi_{k}} \cap G_{k1})^{f_{i}})$. Continuing in this manner the support of $\phi_{j}$ is contained in

%$$\prod_{k = 1}^{s}  G_{\omega_k} (\cap_{i = 1}^{s} (S^{\pi_{k}} \cap G_{kj})^{f_{i}}),$$
%and always contains $(G_{\omega_1} , ...,G_{\omega_s})$. Since the support of $\phi_{j}$ has by assumption the same cardinality as $S$ we have $S \subset (G_{1j} , ...,G_{sj})$ for all $j$ and therefore $S \subset (G_{\omega_1} , ...,G_{\omega_s})$; a contradiction.

\end{proof}

\begin{cor}
Let B be a finite abelian group and $G$ be a self-similar
group with orbit-type $(m_{1} > 1, ..., m_{l} > 1, m_{l + 1} = 1, ..., m_{s} = 1)$. Then
there exist proper subgroups $R_{l+1}, ..., R_{s}$ of $G$ such that
$$B^{((G_{\omega_{1}} \setminus G) \times ... \times (G_{\omega_{l}} \setminus G) \times (R_{l + 1} \setminus G) \times ... \times (R_{s} \setminus G))}\rtimes G^{s}$$
is a self-similar group of orbit-type $(|B|.m_{1} ... m_{l}^{s-l+1}, \, 1)$.
\end{cor}

\begin{proof}
For each $l+1 \leq j \leq s$ the restriction $\dot{f}_{j} = f_{j} : \dot{H}_{j} = H_{l} \rightarrow G$ is well-defined. Define the triple $(\dot{\mathbf{m}}, \dot{\mathbf{H
}}, \dot{\mathbf{F}})$ by
$$\dot{\mathbf{m}} = (m_{1}, ..., m_{l}, m_{l}, ..., m_{l}), \dot{\mathbf{H}} = \{H_{1}, ..., H_{l}, \dot{H}_{l + 1} = H_{l}, ..., \dot{H}_{s} = H_{l}\},$$
$$\dot{\mathbf{F}} = \{f_{1}, ..., f_{l}, \dot{f}_{l+1}, ...,  \dot{f}_{s}\}.$$
and $R_{j}$ the parabolic subgroup of $\dot{f}_{j}$ for $j = l+1, ..., s$. Since $G$ has a faithful state-closed representation with respect to the data $(\mathbf{m}, \mathbf{H%
}, \mathbf{F})$, the same holds for the representation with respect to the data $(\dot{\mathbf{m}}, \dot{\mathbf{H
}}, \dot{\mathbf{F}})$; in fact the ${\bf F}$-core of ${\bf H}$ and the $\dot{\bf F}$-core of $\dot{\bf H}$ are the same. Since $m_{i} > 1$ for each $m_{i}$ in $\dot{{\bf m}}$, we apply Theorem E to obtain the result.
\end{proof}

\end{document}